\documentclass{amsart}
\usepackage{graphicx} 

\newtheorem{theorem}{Theorem}[section]
\newtheorem{proposition}[theorem]{Proposition}
\newtheorem{claim}{Claim}

\title{Finiteness of purely cosmetic fillings}

\author{Kazuhiro Ichihara}

\address{Department of Mathematics, College of Humanities and Sciences, Nihon University, 3-25-40 Sakurajosui, Setagaya-ku, Tokyo 156-8550, Japan}
\email{ichihara.kazuhiro@nihon-u.ac.jp}



\date{\today}

\thanks{This work was supported by JSPS KAKENHI Grant Number JP22K03301.}

\begin{document}

\begin{abstract}
A pair of Dehn fillings on a compact, orientable 3-manifold $M$ with a torus boundary $\partial M$ is said to be \emph{purely cosmetic} if the resulting 3-manifolds are orientation-preservingly homeomorphic. 
In this paper, we show that if $\partial M$ is incompressible, then there are only finitely many pairs of purely cosmetic fillings.
\end{abstract}

\maketitle

\section{Introduction}
Attaching a solid torus to a torus boundary of a 3-manifold is called a \emph{Dehn filling}. 
This fundamental operation gives rise to 3-manifolds from a given 3-manifold. 
It is natural to conjecture that two distinct Dehn fillings on a compact, orientable 3-manifold $M$ with an incompressible torus boundary $\partial M$ are not orientation-preservingly homeomorphic. 
For further details of this conjecture, see \cite{GordonICM} and \cite[Problem 1.81A]{KirbyProblems}.
In view of this, a pair of Dehn fillings on $\partial M$ is said to be \emph{purely cosmetic} if the resulting 3-manifolds are orientation-preservingly homeomorphic.
In this paper, we prove the following result.

\begin{theorem}\label{MainThm}
Let $M$ be a compact, orientable 3-manifold whose boundary $\partial M$ is an incompressible torus. 
Then there are only finitely many pairs of purely cosmetic fillings on $\partial M$.
\end{theorem}

We remark that the same does not hold when $\partial M$ is compressible or has additional connected components; see, for example, \cite[Theorem 2.4.3(c)]{CGLS}. 

The theorem is already known in the case where $M$ is a hyperbolic manifold \cite{FPS22} or a Seifert fibered space \cite{Rong}. 
It is also known for the exterior of a knot in the 3-sphere $S^3$ \cite{Hanselman}.



Similar finiteness results were obtained in \cite{Martelli, RieckYamashita}. 
That is, for an incompressible torus boundary of a compact, orientable 3-manifold, there are only finitely many Dehn fillings that yield a fixed closed 3-manifold.

It is also shown, as a corollary of \cite[Theorem~2.8]{BoyerLines}, that the exterior of a knot in an integral homology sphere admits at most two pairs of integral purely cosmetic fillings. 
A generalization of this result is given in \cite[Theorem~2]{IJ2506} for the exterior of a null-homologous knot in a rational homology sphere obtained by Dehn surgery on a knot in the 3-sphere.



\section{Proof}

Throughout the following, let $M$ be a compact, orientable 3-manifold whose boundary $\partial M$ is an incompressible torus. 
For standard terminology in 3-manifold theory, see \cite{JacoBook}. 

If $M$ is reducible, it can be canonically decomposed into a finite number of connected summands. 
It is shown in \cite[Theorem~4.1]{Martelli} that, for a fixed closed 3-manifold, there are only finitely many Dehn fillings on $M$ that yield it. 
Thus, to prove Theorem~\ref{MainThm}, it suffices to consider the case in which $M$ is irreducible. 
In the following, we assume that $M$ is irreducible.



Suppose that $s_1$ and $s_2$ are purely cosmetic filling slopes on $\partial M$. 
That is, the slopes $s_1$ and $s_2$ are represented by the curves on $\partial M$ that are identified with the meridians of the attached solid tori via purely cosmetic Dehn fillings. 
We refer to these fillings as the $s_1$-filling and $s_2$-filling, and denote the resulting closed 3-manifolds by $M(s_1)$ and $M(s_2)$, which are orientation-preservingly homeomorphic.

When the interior of $M$ is a one-cusped hyperbolic 3-manifold, it is proved in \cite[Theorem~1.13]{FPS22} that $s_1 = s_2$, except for a finite set of explicitly given pairs of slopes for $M$. 
When $M$ is a Seifert fibered space, it follows from \cite[Proof of Theorem~1]{Rong} that $s_1 = s_2$. 
See the references for more details.

Thus, the following is our key proposition. 
From this, Theorem~\ref{MainThm} readily follows, together with Thurston's Hyperbolization theorem for Haken manifolds; see \cite{Kapovich} and \cite{Otal} for its proof and background.

\begin{proposition}\label{PropTor}
Let $M$ be a compact, orientable, irreducible 3-manifold containing essential (incompressible and not boundary-parallel) tori. 
Suppose that the boundary $\partial M$ of $M$ is an incompressible torus. 
Then there are only finitely many pairs of cosmetic fillings on $\partial M$.     
\end{proposition}



For 3-manifolds containing essential tori, the following is well-known as the torus decomposition theorem, a particular case of the JSJ decomposition theorem: 
For a compact, irreducible, orientable 3-manifold $M$, there exists a collection of finitely many disjoint incompressible tori such that each component obtained by cutting $M$ along these tori is either atoroidal or a Seifert manifold, and a minimal such collection is unique up to isotopy. 
See \cite{JacoShalen} and \cite{Johannson} for the original articles, and \cite{NS} or \cite{JacoBook} for further explanations.


Following the arguments in \cite{Martelli}, instead of the torus decomposition, we consider the \emph{geometric decomposition} of $M$. 
This is obtained by replacing each torus in the collection of tori from the torus decomposition that bounds a twisted interval bundle over a Klein bottle $K$ with the core $K$. 
Then, we have the following from \cite[Proposition~3.3]{Martelli}. 
The pieces obtained by cutting $M$ along the finite collection $\mathcal{T}$ of tori and Klein bottles in the geometric decomposition satisfy the following: 
\begin{enumerate}
\item every Seifert block fibers over an orbifold $\Sigma$ with (orbifold) Euler characteristic $\chi(\Sigma) < 0$;
\item the fibrations of the blocks adjacent to a torus or a Klein bottle $S \in \mathcal{T}$ do not induce the same fibration on $S$.
\end{enumerate}
We remark that every Seifert block of the decomposition has a unique Seifert fibration, since its base orbifold has negative Euler characteristic.

\begin{proof}[Proof of Proposition~\ref{PropTor}]
Let $M_0, M_1, \dots, M_n$ be the blocks of the geometric decomposition of $M$, i.e., $M = M_0 \cup M_1 \cup \cdots \cup M_n$, where $\partial M \subset \partial M_0$. 
Let $\mathcal{T}$ be the collection of tori and Klein bottles defining the decomposition. 

We will prove the proposition by induction on $n$. 
If $n = 0$, that is, there are no decomposing tori, then the proposition holds by \cite{FPS22} and \cite{Rong}, as observed above. 
Thus, we consider the case $n = t$ under the assumption that the proposition holds for $n \le t-1$.




Assume that there exist infinite sequences $\{ s_i \} = (s_1, s_2, \dots)$ and $\{ s'_j \} = (s'_1, s'_2, \dots)$ of slopes on $\partial M$ such that $s_i \ne s'_i$ and $M(s_i) \cong M(s'_i)$ for every $i$, and derive a contradiction. 
Here, $\cong$ indicates orientation-preservingly homeomorphic. 
After taking subsequences, and exchanging $\{ s_i \}$ and $\{ s'_j \}$ if necessary, we may assume that all elements of $\{ s_i \}$ are mutually distinct. 
Then, by \cite[Theorem~4.1]{Martelli}, $\{ s'_j \}$ also contains infinitely many elements. 
Thus, after taking subsequences, we may assume that $s_i \ne s_{i'}$ if $i \ne i'$ and $s'_j \ne s'_{j'}$ if $j \ne j'$ for the sequences $\{ s_i \}$ and $\{ s'_j \}$. 

We first show the following claim.

\begin{claim}
Any $T \in \mathcal{T}$ becomes compressible in $M(s_i)$ after only finitely many $s_i$-fillings.
\end{claim}




\begin{proof}
Suppose that some $T \in \mathcal{T}$ becomes compressible after $s_i$-fillings for infinitely many $i$. 
Then, by \cite[Theorem~2.4.4]{CGLS}, $\partial M$ and $T$ cobound a cable space, which must be $M_0$. 
(A \textit{cable space} is the exterior of a $(p,q)$-cable of the core of $V$, where $p$ and $q$ are coprime integers with $q \ge 2$.) 
Furthermore, by comparing the number of blocks in the geometric decomposition, $T$ also becomes compressible after infinitely many $s'_j$-fillings. 
It follows that $M_0(s_i)$ and $M_0(s'_i)$ are all homeomorphic to a solid torus for such $i$. 
Then, since the gluing maps of $M_0(s_i)$ and $M_0(s'_i)$ to the manifold $M' := M_1 \cup \cdots \cup M_t$ vary with $i$, this implies that the manifold $M'$ admits infinitely many purely cosmetic fillings. 
By the induction hypothesis, this gives a contradiction.
\end{proof}


From this claim, after passing to a subsequence if necessary, all elements in $\mathcal{T}$ remain incompressible in all $M(s_i)$ (and $M(s'_j)$). 

\begin{claim}\label{CL2}
For only finitely many $i$, $M_0(s_i)$ becomes an interval bundle over a torus (actually, the product $T^2 \times I$) or over a Klein bottle $K$. 
\end{claim}



\begin{proof}
Suppose that $M_0(s_i)$ yields interval bundles over a torus (actually, products $T^2 \times I$) or over a Klein bottle $K$ for infinitely many $i$. 
We consider the former case, where some $T, T' \in \mathcal{T}$ become parallel after $s_i$-fillings for infinitely many $i$. 
The latter case can be treated similarly.

As claimed in \cite[Proof of Theorem~3.5]{Martelli}, the product $T^2 \times I$ arises as $M_0(s_i)$ only when $M_0 = P \times S^1$, with $P$ the pair of pants. 
Set a homology basis $(m,l)$ on $\partial M$ by taking $m$ as the boundary of a fixed section of $M_0$ and $l$ as the fiber of $M_0$. 
Then, also as in \cite[Proof of Theorem~3.5]{Martelli}, the slopes $s_i$ are represented by some $q_i \in \mathbb{Z}$ with respect to this basis in the standard way. 
It follows that $M(s_i)$ is realized by removing the product $M_0(s_i)$ and gluing the adjacent block(s) directly via a map that twists $q_i$ times along the fiber of $M_0 = P \times S^1$. 
Thus, for sufficiently large $i$, the product bundle $M_0(s_i)$ gives a decomposing torus of the geometric decomposition of $M(s_i)$.
Under the assumption that $M(s_i) \cong M(s'_i)$, by considering the geometric decomposition, the same situation occurs for $M_0(s'_i)$ and $M(s'_i)$ for sufficiently large $i$. 

Let $T$ be the decomposing torus of $M(s_i)$ arising from the product $M_0(s_i)$. 
Then, for sufficiently large $i$, $M(s'_i)$ is obtained by cutting along $T$ in $M(s_i)$ and gluing the adjacent block(s) of the geometric decomposition of $M(s_i)$ directly via Dehn twists along the fiber of $M_0 = P \times S^1$.

If the two boundary tori of $M_0(s_i)$ are glued to obtain both $M(s_i)$ and $M(s'_i)$, then they are torus bundles over the circle. 
Their monodromies differ only by Dehn twists along a simple closed curve. 
Such torus bundles become homeomorphic for only finitely many $i$.

Suppose otherwise. 
Consider an orientation-preserving homeomorphism $h: M(s_i) \to M(s'_i)$. 
From the above arguments, $h$ induces an orientation-preserving self-homeomorphism $h_B := h|_B$ of the block $B$ adjacent to $M_0$ at $T$, such that $h_T := (h|_B)|_T$ is a Dehn twist along the fiber of $M_0 = P \times S^1$.

However, this leads to a contradiction as follows. 
When $B$ is a hyperbolic block, by the Mostow–Prasad rigidity, $h_B$ gives an isometry of the hyperbolic manifold $B - \partial B$. 
Then its restriction to a horotorus corresponding to $T$ is also a Euclidean isometry, which cannot be a Dehn twist along any simple closed curve on $T$. 
When $B$ is Seifert fibered, the restriction $h_T$ of an orientation-preserving homeomorphism $h_B$ of $B$ must preserve the Seifert fibration of $B$, which cannot be a Dehn twist along the fiber of $M_0 = P \times S^1$, since the fibers of $M_0$ and $B$ do not match.
\end{proof}

From these claims, together with the fact that only finitely many $s_i$-fillings can produce new incompressible tori, after passing to a subsequence, we may assume that $\mathcal{T}$ gives the geometric decomposition of both $M(s_i)$ and $M(s'_i)$ for all $i$. 

Then, by the assumption, $M(s_i) = M_0(s_i) \cup M_1 \cup \cdots \cup M_t$ is orientation-preservingly homeomorphic to $M(s'_i) = M_0(s'_i) \cup M_1 \cup \cdots \cup M_t$ for all $i$. 
Since the geometric decomposition is unique up to isotopy, this implies that either $M_0(s_i) \cong M_k$ and $M_0(s'_i) \cong M_{k'}$ for some $k,k' \in \{1, \dots, t\}$, or $M_0(s_i) \cong M_0(s'_i)$ for each $i$.




We now obtain the following.

\begin{claim}
$M_0$ is not hyperbolic.  
\end{claim}

\begin{proof}
Suppose that $M_0$ is hyperbolic. 
Suppose that $M_0(s_i) \cong M_k$ for some $k \in \{1, \dots, t\}$ for infinitely many $i$. 
By the famous Hyperbolic Dehn Surgery Theorem, there are only finitely many exceptional fillings, i.e., fillings yielding non-hyperbolic manifolds, on $\partial M \subset \partial M_0$. 
It follows that $M_0(s_i) \cong M_k$ for infinitely many $i$ implies that $M_k$ is hyperbolic. 
However, since $vol(M_0(s_i)) \to vol(M_0)$ and $vol(M_k) < vol(M_0)$ (see \cite{Martelli}, for example), $M_0(s_i) \cong M_k$ can hold only for finitely many $i$, contradicting the assumption. 
Thus, for infinitely many $i$, $M_0(s_i) \cong M_0(s'_i)$ must hold. 
However, this is impossible by the argument given in \cite[Theorem~1.13]{FPS22}. 
Actually, the theorem considers only the case where the manifold is one-cusped, but by using \cite[Theorem~7.29]{FPS22} instead of \cite[Theorem~7.30]{FPS22}, we see that the same statement holds for the multiple cusp case. 
\end{proof}




We thus consider the case where $M_0$ is Seifert fibered. 
This implies that $M_0(s_i)$ and $M_0(s'_i)$ are also Seifert fibered for all $i$, after passing to a subsequence. 
Note that, at this stage, we do not yet know whether $M_0(s_i) \cong M_0(s'_i)$. 

To proceed further, we prepare some settings. 
Note that an $S^1$-bundle over a surface is obtained from $M_0$ by removing singular fibers. 
In other words, $M_0$ is obtained from the $S^1$-bundle by Dehn fillings along some slopes on the torus boundaries of the bundle. 
We fix a section $F_0$ of the $S^1$-bundle, and set a homology basis $(m_0, l_0)$ on a torus $\partial M \subset \partial M_0$ by taking $m_0$ as the homology class of $\partial F_0 \cap \partial M$ and $l_0$ as that of the fiber of the $S^1$-bundle. 
Then, in the usual way, the slopes $s_i$ and $s'_j$ are parametrized by the set $\mathbb{Q} \cup \{\infty (= 1/0)\}$, where $\infty$ corresponds to the fiber of the $S^1$-bundle. 
Set $s_i = q_i / p_i$ and $s'_j = q'_j / p'_j$ with $p_i, p'_j \ge 0$ for each $i$ and $j$.



Recall that we assume there is an orientation-preserving homeomorphism 
\[
h_i : M(s_i) = M_0(s_i) \cup M_1 \cup \cdots \cup M_t \to M(s'_i) = M_0(s'_i) \cup M_1 \cup \cdots \cup M_t
\]
for each $i$. 



\begin{claim}\label{CL4}
For sufficiently large $i$, $h_i(M_0(s_i)) = M_0(s'_i)$, or $h_i(M_0(s_i)) = M_l$ and $h_i(M_l) = M_0(s'_i)$ for some $l \in \{1, \dots, t\}$ with $M_0 \cap M_l \ne \emptyset$.
\end{claim}

\begin{proof}
Since the geometric decomposition is unique up to isotopy, for each $i$ we have either $h_i(M_0(s_i)) = M_0(s'_i)$, or $h_i(M_0(s_i)) = M_k$ and $h_i(M_{k'}) = M_0(s'_i)$ for some $k, k' \in \{1, \dots, t\}$.
For each $k \in \{ 1, \dots , t \}$, if $h_i(M_0(s_i)) = M_k$ for only finitely many $i$, then $h_i(M_0(s_i)) = M_0(s'_i)$ must hold for infinitely many $i$, and hence the claim follows. 
Suppose instead that, for some $k \in \{ 1, \dots , t \}$, $h_i(M_0(s_i)) = M_k$ for infinitely many $i$. 
After passing to a subsequence, we may assume that, for some $k'$, $h_i(M_{k'}) = M_0(s'_i)$ also holds for these $i$.

With $s_i = q_i/p_i$, since $M_k$ has only finitely many singular fibers, the set $\{p_i\}$ must be bounded; that is, $\{p_i\}$ is a finite set of non-negative integers. 
It then follows that the set $\{q_i\}$ is unbounded, with $|q_i| \to \infty$. 
Similarly, we have $|q'_i| \to \infty$.


Now, let us assign a non-negative integer $\Delta_T$ to each member $T$ of the collection $\mathcal{T}$ of tori and Klein bottles that define the geometric decomposition of $M^i := M(s_i) \cong M(s'_i)$.  
(The idea is inspired by the arguments in \cite{Martelli}.)

First, for each block $N$ of the geometric decomposition of $M^i$ and each boundary torus $T \subset \partial N$, we choose a finite set of \emph{preferred slopes}, containing at least two elements, as follows.  
If $N$ is hyperbolic, a preferred slope is defined to be a slope of shortest or second-shortest length in one cusp section.  
If $N$ is Seifert fibered, we take a section $F$ of the $S^1$-bundle obtained from $N$ by removing neighborhoods of the singular fibers.  
Using this $F$, we set a homology basis $(m,l)$ for each boundary torus $T \subset \partial N$ by taking $m$ as the homology class of $\partial F \cap T$ and $l$ as that of the fiber.  
Then $N$ is obtained from the $S^1$-bundle over $F$ by Dehn fillings along tuple of slopes.  
We assume that the section $F$ is chosen so that each filling slope $q/p$ satisfies $0 < q/p < 1$.  
A preferred slope on $T \subset \partial N$ is then defined as the fiber of $N$ or $\partial F \cap T$.  

Let $T$ be a member of $\mathcal{T}$.  
If $T$ is a torus, we define $\delta_T$ as the maximum of the distance $\Delta(s_1, s_2)$, where $s_1$ and $s_2$ are preferred slopes on $T = \partial N_1$ and $T = \partial N_2$ for the two blocks $N_1$ and $N_2$ adjacent to $T$, respectively. 
If $T$ is a Klein bottle, its neighborhood $W$ admits two fibrations, and we define $\delta_T$ as the maximum of $\Delta(s_1, s_2)$ on the torus $\partial W$, where $s_1$ is a preferred slope of the block adjacent to $W$ at $T$ and $s_2$ is one of the slopes of the two fibers of $W$.  
The quantities $\delta_T$ depend on the choice of sections in the Seifert blocks of $M^i$.  
By varying the sections, we minimize the total sum of the $\delta_T$'s for all $T \in \mathcal{T}$.  
Under this restriction, we define $\Delta_T$ for $T \in \mathcal{T}$ in $M^i$ as the maximum of the $\delta_T$'s associated with $T$.  
Note that each $\delta_T$ has finite ambiguity due to the choice of sections in each Seifert block, but $\Delta_T$ is uniquely determined for $T$.  

Now, since $|q_i| \to \infty$ and $|q'_i| \to \infty$, any member $T \in \mathcal{T}$ attaining the maximal $\Delta_T$ must lie in $M_0(s_i)$ and also in $M_0(s'_i)$ within $M^i$ for sufficiently large $i$.  
(There may be several components in $\mathcal{T}$ attaining the same maximal value.)
We observe that if $\Delta_T$ is maximal among those for the elements of $\mathcal{T}$, then $\Delta_{h_i(T)}$ is also maximal. 
This follows from the definitions of $\Delta_T$, together with the uniqueness of the Seifert block fibrations and the geometric decomposition.

By the observation above, we conclude that for sufficiently large $i$, either $h_i(M_0(s_i)) = M_0(s'_i)$, or $h_i(M_0(s_i)) = M_l$ and $h_i(M_l) = M_0(s'_i)$ with $M_0 \cap M_l \ne \emptyset$ for some $l \in \{1, \dots, t\}$.
\end{proof}

From this claim, for sufficiently large $i$, we have
\begin{itemize}
    \item[(i)] $h_i(M_0(s_i)) = M_0(s'_i)$, or 
    \item[(ii)] $h_i(M_0(s_i)) = M_l$ and $h_i(M_l) = M_0(s'_i)$.
\end{itemize}

Consider Case (i). 
Since $M_0$ has only finitely many singular fibers, the denominators $p_i$ in $s_i = q_i/p_i$ and $p'_i$ in $s'_i = q'_i/p'_i$ must be equal for sufficiently large $i$. 
Then, to have $M_0(s_i) \cong M_0(s'_i)$, we must have $q_i \equiv q'_i \pmod{p}$. 
This implies that the restriction of $h_i$ to some torus $T \subset \partial M_0(s_i)$ is realized by Dehn twists along the fiber of $M_0$ for sufficiently large $i$. 
However, this contradicts the fact that $h_i$ is a homeomorphism between $M(s_i)$ and $M(s'_i)$, by the arguments in the proof of Claim~\ref{CL2} (last paragraph).


Consider Case (ii). 
In this case, 
it follows that 
$M_0(s_i) \cong M_l \cong M_0(s'_i)$. 
Thus, as in the case (i), $p_i$ in $q_i/p_i =s_i$ and $p'_i$ in $q'_i/p'_i =s'_i$ must be equal for sufficiently large $i$. 
Again, to obtain $M_0(s_i) \cong M_0(s'_i)$, we have $q_i \equiv q'_i \pmod{p}$. 
This implies that the restriction of $h_i$ on $T \subset \partial M_0(s_i)$ is realized by Dehn twists along the fiber of $M_0$, and also the restriction on $T \subset \partial M_l$ is realized by Dehn twists along the fiber of $M_l$ for sufficiently large $i$. 
However, it is absurd since the two fibers of Seifert blocks adjacent to $T$ do not fit together. 
\end{proof}

\section*{Acknowledgments}
The author would like to thank Kimihiko Motegi and Hirotaka Akiyoshi for useful discussions.

\bibliographystyle{amsalpha}

\bibliography{ref}

\providecommand{\bysame}{\leavevmode\hbox to3em{\hrulefill}\thinspace}
\providecommand{\MR}{\relax\ifhmode\unskip\space\fi MR }
\providecommand{\MRhref}[2]{%
  \href{http://www.ams.org/mathscinet-getitem?mr=#1}{#2}
}
\providecommand{\href}[2]{#2}
\begin{thebibliography}{CGLS87}

\bibitem[BL90]{BoyerLines}
Steven Boyer and Daniel Lines, \emph{Surgery formulae for {Casson}'s invariant and extensions to homology lens spaces}, J. Reine Angew. Math. \textbf{405} (1990), 181--220 (English).

\bibitem[CGLS87]{CGLS}
Marc Culler, C.~McA. Gordon, J.~Luecke, and Peter~B. Shalen, \emph{Dehn surgery on knots}, Ann. Math. (2) \textbf{125} (1987), 237--300 (English).

\bibitem[FPS22]{FPS22}
David Futer, Jessica~S. Purcell, and Saul Schleimer, \emph{Effective bilipschitz bounds on drilling and filling}, Geom. Topol. \textbf{26} (2022), no.~3, 1077--1188 (English).

\bibitem[Gor91]{GordonICM}
Cameron~McA. Gordon, \emph{Dehn surgery on knots}, Proceedings of the international congress of mathematicians (ICM), August 21--29, 1990, Kyoto, Japan. Volume I, Tokyo etc.: Springer-Verlag, 1991, pp.~631--642 (English).

\bibitem[Han23]{Hanselman}
Jonathan Hanselman, \emph{Heegaard {Floer} homology and cosmetic surgeries in {{\(S^3\)}}}, J. Eur. Math. Soc. (JEMS) \textbf{25} (2023), no.~5, 1627--1670 (English).

\bibitem[IJ25]{IJ2506}
Kazuhiro Ichihara and In~Dae Jong, \emph{Cosmetic surgeries on knots in homology spheres and the {Casson}--{Walker} invariant}, Preprint, {arXiv}:2506.02682 [math.{GT}] (2025), 2025.

\bibitem[Jac80]{JacoBook}
William Jaco, \emph{Lectures on three-manifold topology}, Reg. Conf. Ser. Math., vol.~43, American Mathematical Society (AMS), Providence, RI, 1980 (English).

\bibitem[Joh79]{Johannson}
Klaus Johannson, \emph{Homotopy equivalences of 3-manifolds with boundaries}, Lect. Notes Math., vol. 761, Springer, Cham, 1979 (English).

\bibitem[JS79]{JacoShalen}
William~H. Jaco and Peter~B. Shalen, \emph{Seifert fibered spaces in 3-manifolds}, Mem. Am. Math. Soc., vol. 220, Providence, RI: American Mathematical Society (AMS), 1979 (English).

\bibitem[Kap01]{Kapovich}
Michael Kapovich, \emph{Hyperbolic manifolds and discrete groups}, Prog. Math., vol. 183, Boston, MA: Birkh{\"a}user, 2001 (English).

\bibitem[Kir97]{KirbyProblems}
Rob Kirby, \emph{Problems in low-dimensional topology. ({Edited} by {Rob} {Kirby})}, Kazez, {William} {H}. (ed.), {Geometric} topology. 1993 {Georgia} international topology conference, {August} 2--13, 1993, {Athens}, {GA}, {USA}. {Providence}, {RI}: {American} {Mathematical} {Society}. {AMS}/{IP} {Stud}. {Adv}. {Math}. 2(pt.2), 35-473 (1997)., 1997.

\bibitem[Mar05]{Martelli}
Bruno Martelli, \emph{Links, two-handles, and four-manifolds}, Int. Math. Res. Not. \textbf{2005} (2005), no.~58, 3595--3623 (English).

\bibitem[NS97]{NS}
Walter~D. Neumann and Gadde~A. Swarup, \emph{Canonical decompositions of 3-manifolds}, Geom. Topol. \textbf{1} (1997), 21--40 (English).

\bibitem[Ota96]{Otal}
Jean-Pierre Otal, \emph{Le th{\'e}or{\`e}me d'hyperbolisation pour les vari{\'e}t{\'e}s fibr{\'e}es de dimension 3}, Ast{\'e}risque, vol. 235, Paris: Soci{\'e}t{\'e} Math{\'e}matique de France, 1996 (French).

\bibitem[Ron93]{Rong}
Yongwu Rong, \emph{Some knots not determined by their complements}, Quantum topology. Based on an AMS special session on topological quantum field theory, held in Dayton, OH, USA on October 30-November 1, 1992 at the general meeting of the American Mathematical Society, Singapore: World Scientific, 1993, pp.~339--353 (English).

\bibitem[RY16]{RieckYamashita}
Yo'av Rieck and Yasushi Yamashita, \emph{Cosmetic surgery and the link volume of hyperbolic 3-manifolds}, Algebr. Geom. Topol. \textbf{16} (2016), no.~6, 3445--3521 (English).

\end{thebibliography}

\end{document}